\documentclass{amsart}
\usepackage{amssymb}
\usepackage{amsmath}
\usepackage{amsfonts}
\newtheorem{theorem}{Theorem}[section]
\newtheorem{lemma}[theorem]{Lemma}
\newtheorem{proposition}[theorem]{Proposition}
\newtheorem{corollary}[theorem]{Corollary}
\theoremstyle{definition}
\newtheorem{definition}[theorem]{Definition}
\newtheorem{example}[theorem]{Example}

\theoremstyle{remark}
\newtheorem{remark}[theorem]{Remark}

\numberwithin{equation}{section}

\DeclareMathOperator{\tr}{tr}                  
  
\begin{document}

\title{ Geometric Weil representations for star-analogues of  $SL(2, k)$}

\author{Luis Guti\'errez-Frez}
\address{Instituto de Matem\'aticas, Campus Isla Teja s/n, Valdivia, Chile  }
\curraddr{ }
\email{luisgutierrezfrez@gmail.com}
\thanks{All authors were partially supported by Fondecyt Grant  \#1070246}

\author{Jos\'e Pantoja}
\address{Instituto de Matem\'aticas, Pontificia Universidad Cat\'olica de Valpara\'iso, Valpara\'iso, Chile  }
\email{jpantoja@ucv.cl}
\thanks{ }

\author{Jorge Soto-Andrade}
\address{Depto. Matem\'aticas, Facultad de Ciencias, Universidad de Chile, Santiago, Chile}
\email{sotoandrade@u.uchile.cl}
\subjclass[2000]{Primary 20C33, 20F05; Secondary  20C25, 15A33, 16W10,17B37}

\date{May 15, 2010 }

\dedicatory{ }

\keywords{Weil representation, Bruhat presentation, generators and relations, star-analogues, geometric Gauss sum, classical groups, contraction, connection}

\begin{abstract}
We present here an elementary geometric approach to the construction of Weil representations of the star-analogues $SL_\ast(2,A),  \ A$ a ring or algebra with involution $\ast$, of the group SL(2, k), k a field, reminiscent of the quantum groups $SL_q(2,A)$.  We review as well the elementary construction of Weil representations for these groups via generators and relations, which uses the Bruhat presentation available in many cases. We compare the representations obtained by both methods in the non - classical case of the finite truncated  polynomial algebra $A_m$ of degree $m$  with its canonical involution and obtain the analogue of the Maslov Index in this case.
\end{abstract}

\maketitle

\section {Introduction}

The introduction of star-analogues (also written ``$\ast$-analogues") of classical groups may be motivated by the attempt to construct a theory of \ ``non conmutative determinants'' (i.
e. to define some sort of determinant for square matrices with non commuting
entries). This problem has a long story, going back to Cayley and Dieudonn\'e \cite{D} and
undergoing new developpments in the hands of I. M. Gelfand and collaborators in the
90's \cite{GR}.

To approach this old problem, we may try to settle for a case
which lies half way between commutativity and non-commutativity, to
wit, the case in which non commutativity is ``controlled'' by an involution
in the coefficient ring  $A,$ where the entries of our matrices lie. More
precisely, we assume that  $A$ admits an involutive anti-automorphism,
denoted  $T:a\mapsto a^{\ast }.$ We say in this case that $A$ is an {\em involutive ring}, or a {\em ring with involution}. We may think then of \ $T$ \ as a
measure of the lack of commutativity of the multiplication \ $m\;$ in \ $\ A.$
In fact, when $A$ is an involutive algebra, if we write \ $m:A\otimes A\rightarrow A$ for the multiplication
of \ $A$ \ and \ \ $S:A\otimes A\rightarrow A\otimes A$ \ \ for the ``flip''
\ $x\otimes y\longmapsto y\otimes x,$ we see that \ $T$ \ ``transforms'' \ $m
$ $\ $into $\ m\circ S$ as follows: \ \ \ \ \ $m\circ S=T\circ m\circ (T\otimes T)^{-1}.$
This is quite analogous to the way in which the $R-$matrix ''controls'' the
lack of co-commutativity in a quantum group, i. e. \ \ \ $S\circ \Delta =$ $%
\mathcal{I}_{R}\circ \Delta ,$\ \ where $\Delta $ stands for the
comultiplication in the corresponding Hopf algebra and   \ $\mathcal{I}_{R}
$ \ denotes conjugation by the \ $R-$matrix, which is an invertible element
in    \ $A\otimes A.$\ 

So, if we consider, to begin with,\ $2\times 2$ \ matrices over a ring \ $A$ \ with
involution, we may expect to be able to define a non-commutative (or
rather, a \ $\ast -$commutative)\ \ $\ast -${\em determinant} for matrices
whose entries \ ``$\ast -$commute'', by an expression of the form \ \ $%
ad^{\ast }-bc^{\ast },$ or the like.\ \ \  

This is indeed the case, as we explain in section 2 below, where the groups \ $GL_{\ast
}(2,A)$ and $SL_{\ast }(2,A)$ are introduced as in \cite{PSAjalg}, in a way reminiscent of  the
quantum groups $\ GL_{q}(2,A)$ and $SL_{q}(2,A).$

Notice that by taking the most familiar example of a non-conmutative  involutive
ring, to wit $A=M(n,k), k$ a field, with the transpose mapping as involution,
  our $SL_{\ast }(2,A)$ \ is none other than the symplectic
group \ $Sp(2n,k),$ and that $GL_{\ast }(2,A)$\ is the associated
symplectic similitude group \ $GSp(2n,k) $ \cite{sa1}.
Varying the ring with involution, we may obtain other classical and non-classical groups as well  \cite{PSAcom, lucho}.

 This viewpoint suggests among other things that these groups may admit presentations, in terms of generators and relations,  that are   
natural $\ast -$analogues of the well known
{\em Bruhat presentation} for \ $GL(2,k)$ and \ $SL(2,k),$ over a base field
\ $k $ \cite{sa1}.  
 
This is indeed the case when the involutive ring $A$ is a full matrix ring over a field, with the
transpose as involution \cite{sa1},  or a truncated polynomial ring with the canonical involution \cite{lucho}, besides the very classical case of $\mathbb Z$ (with the trivial involution).  As described below the existence of this sort of presentation seems to be
closely related to the existence of a weak non commutative $\ast$ - analogue of the
euclidean algorithm in the base involutive ring. 

A first step to establish the existence of Bruhat presentation for general $\ast$-analogues of $GL_{\ast }(2,A)$ \ and $\ SL_{\ast }(2,A)$, was already accomplished in  \cite{PSAjalg}, where a {\em Bruhat decomposition} for these groups was obtained in the case of an artinian
involutive base ring $A$.  Later, the classical Bruhat presentation of $GL(2,k)$
was extended to the case of an artinian simple involutive $A$ in \cite{P}. Then
the existence of a Bruhat presentation for $GL_{\ast }(2,A)$ \ and $\ SL_{\ast }(2,A)$ was proved in \cite{PSAcom} for a wide class of involutive rings  $A$, namely, 
those admitting an $\ast -$analogue of the euclidean algorithm for coprime elements, called  here ``weakly euclidean rings with involution", that includes the artinian simple rings considered in  \cite{P}.   

Turning now to group representation theory, we recall that in  the special case of $A = M(n, k) $    with the transpose   involution, in \cite{sa1}  the  Bruhat presentation obtained was put to use to construct in an elementary way {\em Weil representations} for the symplectic groups $Sp(2n,k)$ \ and their associated symplectic similitude groups  \ $GSp(2n,k), \ k $ a finite field, and then to obtain - by decomposition - all their irreducible representations when  $ n\leq 2 $.  The method  employed was to define the Weil operators for the Bruhat generators of the group  and then  checking that the Bruhat relations among these generators were preserved by the given operators.  The crucial point of the verification was the calculation of  Gauss sums associated to quadratic spaces of even rank over $k$.

Recently, in \cite{lucho} a Weil representation was constructed following this line of thought for the non classical group    $SL_{*}(2,A_m)$,  where  $A_m=\mathbb{F}_q[x]/\hspace{-1.5mm}\left\langle x^m\right\rangle$ endowed with the involution determined by  $ x^\ast = -x$, by proving first that this non semi-simple involutive ring  $A_m$ is weakly euclidean and calculating the relevant Gauss sums.
We recall that   for the ring $A=\mathbb{Z}/p^{n}\mathbb{Z}$, with trivial involution,  a Weil representation for $SL(2,A)$ has been constructed by
a different method in \cite{cms}, where  the decomposition problem is solved.
This problem however remains open for most Weil representations so
constructed.

We should point out however that a different approach to the   construction of Weil representations, which is  also elementary but more geometric in nature,  was sketched first in 
\cite{sa2} for  $SL(2, k)$  and later applied in  \cite{PSAcr} to recover the classical Weil representation in the case of   \ $Sp(2n,k), \ k $ a finite field, as well as the generalized Weil representations constructed with the help of {\em Grassmann-Heisenberg groups} for  $SL(n,k) $, $n$  even, in  \cite{PSAast}.  For another  different  geometric, albeit non elementary  construction, see  \cite{guha}.

Our aim in this paper is to construct in a elementary geometric way Weil representations for
our   \ $SL_{\ast }(2,A)$ \ groups and to compare them with the Weil representations hitherto constructed via generators and relations with the help of a Bruhat presentation.  The content of this paper is as follows.

In section 2, we introduce the $\ast$-analogues    $\ GL_{\ast
}(2,A)$ \ and $\ SL_{\ast }(2,A)$ and give several specific examples.

In section 3 we explain how a Bruhat presentation for our groups follows from the existence of a $\ast $-analogue of a weak euclidean algorithm in the corresponding base involutive ring  $A$ and recall the results obtained so far along this line.

In section 4, we present first the general method of contraction of a $\ G$-Hilbert 
bundle along a given  $G$-equivariant connection, to construct  representations of a finite group  $G$ and give complete proofs of some results announced in the note \cite{PSAcr}.  We then specialize to the case relevant to us here,  of a finite $k$-algebra with involution, $k$ a finite field, for which a Lagrangian Hilbert bundle is constructed, endowed with a natural $G$-equivariant connection.  We contract this bundle over a base point, along the given connection, to construct a projective unitary (generalized) Weil representations for $ G  = SL_\ast(2,A) $ associated to a given self-dual $A$-module and we express its 2-cocycle in terms of a ``geometric Gauss sum'', which gives the analogue of the Maslov Index in this case (see \cite{lv, JP}).

Section 5 address the case of a non semi-simple involutive base ring $A$ with nilpotent radical, to wit the case of the truncated polynomial ring       
$ A_{m}=k[x]/\left\langle x^{m}\right\rangle  $   of  nilpotency degree $m$ over the finite base field $k$.
We recall the construction of a Weil representation of  $G = SL_\ast(2,A) $  via Bruhat generators and relations, following \cite{lucho}.  We apply then the geometric method of contraction of a suitable Lagrangian bundle for $G$ to construct a projective Weil representation of  $G$, that we compare with the true representation previously constructed, to find that they coincide up   
  to multiplication by a ``correcting'' 1-cocycle whose coboundary is the 2-cocycle of our projective geometric Weil representation.

\section{The groups $ {GL}_{\ast }(2,A)$  and $
{SL}_{\ast }(2,A)$}

We recall the definition  of  the groups   $GL_*(2,k)$   and  $
{SL}_{\ast }(2,A)$  introduced in   \cite{PSAjalg}.
 Let $A$ be a unitary ring with an  involution,  i.e.  an involutive anti-automorphism of $A$, denoted  $ \ast:  a\mapsto a^{\ast }$.
 
Let $GL_{\ast }(2,A)$ be the set of matrices $g=\left( 
\begin{array}{cc}
a & b \\ 
c & d
\end{array}
\right), $ $\ a,b,c,d\in A$ such that $ab^{\ast }=ba^{\ast },cd^{\ast
}=dc^{\ast },a^{\ast }c=c^{\ast }a,b^{\ast }d=d^{\ast }b,ad^{\ast }-bc^{\ast
}=a^{\ast }d-c^{\ast }b \in Z_{s}(A)^{\times }$,   where $Z_{s}(A)^{\times }$%
denotes the group of the symmetric, central, invertible elements of $A.$

 Let   $\det_{\ast }$, be the function on $GL_{\ast
}(2,A)$ defined by $\det_{\ast }(g)=ad^{\ast }-bc^{\ast }.$  We call  $\det_{\ast }$  the $\ast -${\em determinant}  of    $GL_{\ast}(2,A)$.

The set $GL_{\ast }(2,A)$ is a group under matrix multiplication and $\det_{\ast }$ is an epimorphism of groups  from $GL_{\ast }(2,A)$ onto   $ Z_{s}(A)^{\times }$ (see \cite{PSAjalg}).

We define then  $SL_{\ast }(2,A)$ as the kernel of the epimorphism  $\det_{\ast }. $  

\subsection{Examples}

\begin{enumerate}
\item[i.] Let $A=M(n,k)$, $k$ a field and let $\ast $ be the usual transpose
mapping. Then $SL_{\ast }(2,A)=Sp(2n,k)$.

\item[ii.] Let $A=M(2m,k)$, $k$ a field and let $\ast $ the be the adjoint mapping
with respect to a non-degenerate skew symmetric form on \ $\ \ V=k^{2m}.$ Then \ $SL_{\ast
}(2,A)=O_{+}(2m,k),$ the orthogonal group of the hyperbolic quadratic
form of rank \ $2m.$

\item[iii.]  Let $A=k\mathbf{[}G]$, $k$ a field, $G$ a finite group and $\ast $ the
involution on $A$ defined by $g^{\ast }=g^{-1}$. 
 
\noindent
Notice that this example affords a non-classical group  $SL_{\ast }(2,A)$  in the modular case, when the group algebra  $k\mathbf{[}G]$       is not semi-simple.

\noindent
More generally, one could take $A$ $\ $to be a Hopf algebra, with the
antipode as involution (when this is the case).

\item[iv.] The doubling construction \cite{PSAjalg}: 
Let $A_{1}$ and $A_{2}$ be two rings with identity and  $\varphi
:A_{1}\rightarrow A_{2}$ be an anti-isomorphism.
Let $A=A_{1}\oplus A_{2}$. We define an involution $\ast $ \ in $A$ \ by \
$
(x_{1},x_{2})^{\ast }=\left( \varphi ^{-1}\left( x_{2}\right) ,\varphi
(x_1\right) ),$ \ for \ \ $x_{i}\in A_{i}$. \ Then we have:
$SL_{\ast }(2,A)\cong Gl(2,A_{1})$ 
 
\noindent
Notice  that to any unitary ring \ $R$ \ we may associate its \
''involutive double'', i. e. the involutive ring \ $\mathcal{D}(R)$
generated by \ $R,$ as follows: just apply the preceding construction to \ \
\ $A_{1}=R, A_{2}=R^{op},$ the opposite ring to \ $R,$ and \ 
$\varphi = Id:R\rightarrow \ R^{op}.$ Then  $GL(2,R)  \cong  SL_{\ast }(2,\mathcal{D}(R)).$
 \end{enumerate} 
 
\section{Bruhat Presentation of  $G = SL_{\ast }(2,A) $}
 
The classical euclidean algorithm implies that if  $a,c\in \mathbb{Z}$ are
such that  \mbox{$\mathbb{Z}a\mathbb{+Z}c=  \mathbb{Z},$} \ then there exists an
euclidean chain

$a=s_{0}c+r_{0}$

$c=s_{1}r_{0}+r_{1}\ $

$r_{0}=s_{2}r_{1}+r_{2}$

$r_{1}=s_{3}r_{2}+r_{3}$

...

$r_{n-2}=s_{n}r_{n-1}+r_{n}$ \ ,

such that \ $r_{n}\in\mathbb{Z}^{\times}$ (the multiplicative group of $%
\mathbb{Z}$), and conversely.

  In terms of \ the group \ $SL(2,\mathbb{Z)}$ this may be
re-interpreted as follows:

\qquad Let \ \ $u(s)=\left( 
\begin{array}{cc}
1 & s \\ 
0 & 1
\end{array}
\right) $ \ \ for \ $s\in\mathbb{Z}$, \ \ $w=\left( 
\begin{array}{cc}
0 & 1 \\ 
-1 & 0
\end{array}
\right) ,\quad$

$h(a)=\left( 
\begin{array}{cc}
a & 0 \\ 
0 & a^{-1}
\end{array}
\right) \quad$    for \ $a\in\mathbb{Z}^{\times}.$ \\ 
Then \ for every \ \ $
g=\left( 
\begin{array}{cc}
a & b \\ 
c & d
\end{array}
\right) \in SL(2,\mathbb{Z)\quad}$
there exists  a double sequence  \ $r_{0},s_{0}, r_{1},s_{1},...,r_{n+1},s_{n+1}$ \ of elements of  \ $\mathbb{Z},$ such that \ 

$$wu(-s_{n+1})...wu(-s_{1})wu(-s_{0})g=\left( 
\begin{array}{cc}
\ast & \ast \\ 
0 & \ast
\end{array}
\right) =h(e)u(t),$$
\noindent
for suitable \ $e=\pm1, t \in \mathbb{Z},$ so that  $ SL(2,\mathbb Z) $ is generated by  
$ u(s), \ s \in \mathbb Z,  h(a), \ a\in \mathbb Z^\times  \ $ and $w$.

\bigskip

This motivates the following definition, for an involutive ring \ $(A,\ast )$, \ instead of \ $
\mathbb{Z}$.  

\begin{definition}
$ (A, \ast) $ is a weakly euclidean ring  (called  $\ast$-euclidean ring in \cite {PSAcom})
if given $a ,b \in A$ such that $a ^{\ast }b =b ^{\ast
}a $ (i.e.  $a$ and $b$ $\ast$ - commute), $ Aa +Ab =A$ (i. e. $a $ and $b $ are 
 coprime), then there are finite sequences of elements $
s_{0},s_{1},...,s_{n-1}\in A_{s}$ and $r_{1},r_{2},...,r_{n}\in A$ with $
r_{n}\in A^{\times }$ such that

$a =s_{0}b +r_{1}$

$b =s_{1}r_{1}+r_{2}$

.
.
.

$r_{n-2}=s_{n-1}r_{n-1}+r_{n}$

\end{definition}

We observe that if $\left( 
\begin{array}{cc}
a & b \\ 
c & d
\end{array}
\right) \in SL_{\ast }(2,A)$, then $a$ and $c$ are  coprime.

\bigskip

Set now 
$$h(a)=\left( 
\begin{array}{cc}
a & 0 \\ 
0 & (a^{\ast}) ^{-1}
\end{array}
\right),    \ \  \ u(s)=\left( 
\begin{array}{cc}
1 & s\\ 
0 & 1
\end{array}
\right),  \ \ \
w =\left( 
\begin{array}{cc}
0 & 1 \\ 
-1 & 0
\end{array}
\right), 
  $$ 
for   $a \in  A^\times, \ s \in  A_s,$ \ which we call {\em Bruhat generators} for $SL_{\ast }(2,A).$ \ 

Then we may easily check:

\begin{lemma}:

Let $\left( 
\begin{array}{cc}
a & b \\ 
c & d
\end{array}
\right) \in SL_{\ast }(2,A)$ with $c\in A^{\times }$. Then

$\left( 
\begin{array}{cc}
a & b \\ 
c & d
\end{array}
\right) = h(-c^{\ast ^{-1}})u(-c^{\ast }a) w u_(c^{-1}d)$

 \end{lemma}
 
The following theorem is proved in   \cite{PSAcom}
\begin{theorem} \label{eucligen}
If $(A, \ast) $ is a weakly euclidean ring, then the elements
$ \ \   u(s), $\\
 $ (s \in A_s),  h(a), \ (a \in A^\times)$ and  \ $w, $
generate the group $SL_{\ast }(2,A)$.

 \end{theorem} 
 \begin{flushright} 
$ \square$
\end{flushright} 

Notice that if we have a least upper bound for the length of the euclidean
chains associated to pairs of   $\ast$-commuting elements \ $a,c   \in A$, which are 
coprime, then we get a corresponding least
upper bound for the \ $w-$length of the expression of any \ $g\in
SL_\ast(2,A) $ as a word in terms of our generators.
Recall that the $w$-length of an element $g\in G $ is the minimal $j$ such that \mbox{$g\in(BwB)^j$,} where $B$ is the subgroup of $G$ generated by $h(t)$, and $u(b)$, $t\in A_m^{\times}$, $b\in A_m^{s}$ and $B^0=B$.

For instance, in the case of \ $A=M_{n}(k),k$ a commutative field, with the transpose mapping as involution, we have proved \cite{sa1}:

\begin{proposition}

If \ $a,c\in A$ \ satisfy \ \ $a^{\ast }c=c^{\ast }a$ $\ $and$\ $\ $Aa+Ac=A$%
, then there exists a symmetric matrix \ $s\in A,$ such that \ \ $\ a+sc\in
A^{\times }$
\end{proposition}
\begin{flushright} 
$ \square$
\end{flushright} 
From this result it follows that the group $SL_\ast(2, A)$ is generated in this case by its Bruhat generators and that its Bruhat length is 2.

Moreover, in \cite{PSAcom} the following crucial lemma is proved:

\begin{lemma} \label{lemma4comm}
Let $A$ be a simple artinian ring with involution that is either infinite or
isomorphic to the full matrix ring over $\mathbb{F}_{q}$ with $q>3$. Let
$a,b\in A_{s}$ be such that $a,b\notin A^{\times}.$ Then there exists $u\in
A^{\times}\cap A_{s}$ such that $a+u,b - u^{-1}\in A^{\times}$
\end{lemma}
\begin{flushright} 
$ \square$
\end{flushright} 
From here, we are able to prove: 

\begin{theorem} \label{theorem1comm}
With the hypothesis of lemma 4, the group $SL_{\ast }(2,A)$ has a Bruhat presentation, i.e.,
$G=< h(t),u(b),w:   t\in A^{\times},b\in A_{s},\mathcal{R} > $ where $\mathcal{R}$ is the set of relations

1. $ h(t)h(t')=h(tt') $

2. $ u(b)u(b')=u(b+b') $

3. $ w^{2}=h(-1)$

4. $h(t)u(b)=u(tbt^{\ast})h(t)$

5. $w h(t) = h(t^{\ast^{-1}})w$

6. $w u(t^{-1}) w u(t)w u(t^{-1})=h(t)$

where  $t, t' \in A^\times, b, b' \in A^s.$

\end{theorem}
\begin{flushright} 
$ \square$
\end{flushright} 
This theorem generalizes our previous results \cite{PSAjalg, P} concerning the existence of Bruhat presentations for  $G = 
SL_\ast(2, A)$  but leaves open the case of a non-semisimple involutive ring $A$, for instance.

\section{Geometric construction of  Weil representations of   $SL_\ast(2, A)$}

 \subsection{Construction of representations of $G$ by contraction of a  $G$-Hilbert bundle over a base point}

Let $G$ be a finite group and   $ \mathcal H = (E, p, B, \tau) $ a   $G$- Hilbert bundle, with total space $E$, base $B$, projection  $p: E \longrightarrow B$ and $G$-action   $\tau = (\tau^E, \tau^B)$, where 
$\tau^E$ is an action of  $G$ in $E$, $ \tau^B$ is an action of   $G$ in  $B$,  such that
$      p \circ  \tau^E_g   =   \tau^B_g \circ  p  $  for all    $g \in G $.   The fiber  
$ p^{-1}(b) $ above   $b \in B$ will be denoted  by  $E_b$. We also write  $\tau$ instead of  $\tau^E $  or   $\tau^B,$ and also simply
$ \tau^E_g(v)    =   g.v, \ \  \tau^B_g(v) = g.b, $ \ for $g \in G, v \in E, b \in B$. 

Moreover each fiber   $E_b$ is endowed with a (finite dimensional) Hilbert space structure $<  ,   > $ which is preserved by the $G$-action $\tau$. 

\begin{definition}
A $G$- equivariant connection on   a   $G$- Hilbert vector bundle
 $ \mathcal H = (E, p, B, \tau) $ is a family of Hilbert space isomorphisms
 $\Gamma =\left\{ \gamma _{b^{\prime },b} \mid \gamma _{b^{\prime },b}: {E} 
_{b}\rightarrow  {E}_{b'}\right\}_{b, b' \in B}$
such that 
\begin{enumerate}
\item[i.]  $\left\langle \gamma _{ b^{\prime }, b}(f),\gamma _{ b^{\prime
}, b}(h)\right\rangle =\left\langle f,h\right\rangle $ \ \ \  \   \ $(f,h\in  E%
_{b})$
\item [ii.] $\left\langle \gamma _{ b^{\prime }, b}(f),h\right\rangle =\left\langle
f,\gamma _{ b, b^{\prime }}(h)\right\rangle $ \  \  \  $(f \in  E_{ b}, h\in 
 E_{ b^{\prime }})$

\item[iii.]  $\gamma _{ b, b'} \circ \gamma _{ b^{\prime }, b}=  \gamma_{ b, b}=  id_{ E_{ b}}$ \  \  \  \  \ $(b, b' \in B)$

\item [iv.] $\gamma _{b'', b'} \circ \gamma _{ b^{\prime }, b}= \mu_\Gamma(b'',b',b)
 \gamma _{ b'', b}$          \ \ \ \ \  $  (b, b', b''  \in B)$

 \noindent 
for a suitable mapping   $\mu_\Gamma: B\times B \times  B  \rightarrow \mathbb C^\times  $, called the   multiplier of  $\Gamma$.
  
\item[v.]   $\tau _{g}\circ \gamma _{b',b}=\gamma _{\tau_g(b'),\tau_g(b)}\circ
\tau _{g}\medskip $      \ \ \ \    $( b, b' \in B,  g \in  G).$

\end{enumerate}

We say that the connection  $\Gamma$ is flat iff its multiplier  $\mu_\Gamma$ is the constant function  $\mathbf 1$.

\end{definition}

\begin{proposition}
Given a $G$- Hilbert space bundle  $ \mathcal H = (E, p, B, \tau) $ endowed with a $G$-equivariant connection   $\Gamma =\left\{ \gamma _{b^{\prime },b} \mid \gamma _{b^{\prime },b}: {E} 
_{b}\rightarrow  {E}_{b'}\right\}_{b, b' \in B}$, with multiplier $\mu$, we can associate to each point  $b \in B$  a projective unitary representation   $(V_b, \rho^b) $ of   $G$  defined as follows:
\begin{enumerate}
\item[i.] $ V_b = E_b$  \ \ \   as a Hilbert space,  

\item[ii.]   $ \rho^b_g (v) =    \gamma_{b, \tau^B_g(b)}    \tau^E_g (v) $  \ \ \  for all  $  g \in G,  v \in V_b $
\end{enumerate}   
whose cocycle $  c $ is given by 

$   c (g,h) = \mu_\Gamma (b, g.b, gh.b) $    \ \ \      for all    $g, h \in G. $
If $\Gamma$ is flat then $\rho^b$ is a true representation of $G$.

\end{proposition}

The proof is a straightforward calculation.

\begin{flushright} 
$ \square$
\end{flushright} 

\begin{definition}
The representation $(V_b, \rho^b) $ constructed by the proposition above, is called   the representation of $G$ obtained by contraction of the Hilbert bundle $\mathcal H $ over $b$ along the connection  $\Gamma$.

\end{definition}

\begin{remark}
Notice that the linear isomorphisms    $\gamma_{b',b} $ of   $\Gamma$ afford isomorphisms from the representation    $(V_b, \rho^b) $ onto  $(V_{b'}, \rho^{b'}) $, even if  $b$ and   $b'$ do not belong to the same   $G$-orbit in   $B$.
\end{remark}
\bigskip 
We recall next how to construct Weil representations by this contraction procedure,

\subsection{Construction of Weil representations of   $ G = SL_\ast(2, A)$  by contraction of Lagrangian fiber bundles} 
 
Let $ A $ be a finite $k$-algebra with an involution $ \ast$ that fixes the finite base field $k$. 
  We will construct now 
a representation of $G = SL_{\ast }(2,A)$ by  contraction of
a suitable  $G$-fiber bundle along an appropiate $G$-equivariant connection.

 \subsubsection{ Lagrangian bundles for $G$} \label{sslagrangian}
  Let $S$ be an  $A$-left module which is finite dimensional as a  $k$-vector space. Then $S$ is an $A$-right module as well, with $s.a=a^{\ast }.s$ $ \ \ \ (a\in A,s\in S)$. We  suppose given a non-degenerate, $k$-bilinear, symmetric
and $A$-balanced pairing  $\eta :S\times S\rightarrow  k. $  Recall   that $\eta$ is $A$-balanced iff  for all  
 $ a\in A,  s, t \in S,$   we have  $\eta (s.a,t)=\eta (s,a.t)$   i.e.    $\eta (a^\ast.s, t)=\eta (s, a.t),$   so that  $a^\ast$ appears as the adjoint of $a \in A$ with respect to  $\eta$.
 
 We say then  that $(S,\eta )$ is a \textit{
self dual }$A-$\textit{module}.

    We set $W=S\oplus S$  and  we define a symplectic form   $B$ on  $W$ by
$$ B((s,t), (s',t')) = \eta(s,t') - \eta(t,s'))$$  for all $ (s,t), (s',t') \in  W$.  We fix   a non trivial character   $\psi$  of  the additive group  $k^+$ of  $k $  and we put     
  $\chi =  \psi \circ  B.$
  
  We define a \textit{
Lagrangian} $L$  in  $W$ to be a right $A$ - submodule $L$  of $W$ which is maximal totally isotropic for $B$, or equivalently, such that $
L=L^{\perp }$.                                                                

For a complete description of the Lagrangians in the case of $A = M(n,k) $(example i. in section 2),  see \cite{PSAcr}.

Notice that the group $G = SL_{\ast }(2,A)$ acts naturally on $W$ by
left and right matrix multiplication, according to vectors in $W$ being looked upon as column or row vectors, respectively.  We will write  $ g.w $  for the first and   $w.g$ for the second, for $w \in W, g \in G$.  Then in what follows 
$gw $ may mean   $ g.w $  or    $w.g^{-1}$  and   $gL,$   for    $ L \subset  W$ means in any case   $ \{ gw | w \in L \} $.  
 Recalling that the pairing  $\eta$ is $A$-balanced, we check easily the following

\begin{proposition}
We have  $$B (gw,gw^{\prime })= B(w,w^{\prime })$$  
 for all  $g\in SL_{\ast }(2,A);w,w^{\prime }$ $\in W$.
\end{proposition}
 \begin{flushright} 
$ \square$
\end{flushright} 

We will construct now a $G$- Hilbert bundle $\mathcal{H} = (
 \mathcal E, p, \mathcal L, \tau )$ called the {\em Lagrangian  bundle of $G$ associated to  $S$}, as follows:
\begin{enumerate}
\item[i.] $\mathcal L$ is the set of all Lagrangians of $W = S \oplus S;$

\item[ii.]  $\mathcal{E}$ is the disjoint union of the spaces 
$$ \mathcal{E}_{L}  =\left\{ f:W\rightarrow \mathbb{C}\mid f(w+ \zeta )=\chi
(w, \zeta )f(w); \ \ w\in W, \  \zeta \in L \right\}, $$
for  $L \in \mathcal L $, each endowed with 
the inner product given by $$\left\langle f,h\right\rangle =%
\underset{w\in W}{\sum }f(w)\overline{h(w)}$$.

\item[iii.]  $p :\mathcal{E}\rightarrow \mathcal{L}$ is given by $\mathit{p}(f) =L$
if  $f\in $ $\mathcal{E}_{L}$

\item[iv.]  $\tau $ denotes the action of $G$ in $\mathcal{E}$ and $\mathcal{L}$
given by 
$$\left( \tau _{g}(f)\right) (w)=f(g^{-1}w),
\tau _{g}(L)=g(L),$$
for  $g \in G, f \in \mathcal E, w \in W, L \in \mathcal L.$
\end{enumerate}
\bigskip

\begin{theorem}

Assume that $2\in A$ is invertible. Then the
family 
$$\Gamma =\left\{ \gamma _{L^{\prime },L} \mid \gamma _{L^{\prime },L}: {E} 
_{L}\rightarrow  {E}_{L'}\right\}_{L, L' \in \mathcal L}$$

of linear isomorphisms  
  
$$ \gamma _{L',L}(f)(w)=\frac{1}{\sqrt{\left| L\right| \left| L\cap L'\right| }}\underset{ \zeta
^{\prime }\in L'}{\sum }\overline{\chi (w, \zeta ^{\prime })}
f(w+ \zeta ^{\prime }),  \ \  (f \in \mathcal E_L, w \in W) $$  
 is a $G$-equivariant
connection with multiplier 
$$ \mu_\Gamma (L'', L', L) =
\sqrt{\frac{\left| L''\cap L'\right| }{\left| L\cap
L''\right| \left| L'\cap L\right| \left| L\right| }}
S_{W}(L;L',L'')  \  \  \  \ (L, L', L'' \in   \mathcal L),$$

where the geometric Gauss sum   $S_{W}(L;L',L'') $ is given by

$$S_{W}(L;L',L'') = \underset{ \zeta \in L\cap (L' + L'')}{\sum }\chi ( \zeta' +  \zeta'' )$$

where $ \zeta \in L$ $\cap (L'+L'')$ is written as 
$ \zeta ^{\prime }+ \zeta '' $ with $ \zeta ^{\prime
}\in L', \zeta '' \in L''.      $   

In other words, we have,  for all  \   \  $L,   L',  L'' \in \mathcal L$):  
\begin{enumerate}

\item[a)]  $\left\langle \gamma _{L',L}(f),h\right\rangle =\left\langle
f,\gamma _{L,L'}(h)\right\rangle $ \ $($\ $f\in \mathcal{E}_{L},h\in 
\mathcal{E}_{L'})$

\item[b)]   $\left\langle \gamma _{
,L}(f),\gamma _{L^{\prime
},L}(h)\right\rangle =\left\langle f,h\right\rangle $ \ \ $(f,h\in \mathcal{E}%
_{L})$

\item[c)]  $\gamma _{L',L}\circ \gamma _{L',L}=   \gamma_{L, L} = id_{\mathcal{E}_{L}}$

\item[d)]  $\gamma _{L'',L'}\circ \gamma _{L',L}=%
\sqrt{\frac{\left| L''\cap L'\right| }{\left| L\cap
L''\right| \left| L'\cap L\right| \left| L\right| }}%
\mathit{S}_{W}(L;L',L'')\gamma _{L'',L}$ \\
\noindent 
where $$S_{W}(L;L',L'') = \underset{ \zeta \in L\cap (L' + L^{\prime
\prime })}{\sum }\chi ( \zeta ^{\prime }, \zeta '' )$$ 
where $ \zeta \in L \cap (L'+L'')$ \ is written as 
$ \zeta ^{\prime }+ \zeta '' $ with $ \zeta ^{\prime
}\in L', \zeta '' \in L''.$

\item[e)]  $\tau _{g}\circ \gamma _{L',L}=\gamma _{g(L',g(L)}\circ
\tau _{g}  $    \ \ \  ( $ g \in G $).

\end{enumerate}
\end{theorem}

\begin{proof}

We prove first a).

$\left\langle \gamma _{L',L}(f),h\right\rangle =\underset{w\in W}{%
\sum }\gamma _{L',L}(f)(w)\overline{h(w)}$

$=\underset{w\in W}{\sum }\frac{1}{\sqrt{\left| L\right| \left| L\cap
L'\right| }}\underset{ \zeta ^{\prime }\in L'}{\sum }%
\overline{\chi (w, \zeta ^{\prime })}f(w+ \zeta ^{\prime })\overline{%
h(w)}$

=$\underset{z\in W}{\sum }\frac{1}{\sqrt{\left| L\right| \left| L\cap
L'\right| }}\underset{ \zeta ^{\prime }\in L'}{\sum }%
\overline{\chi (z- \zeta ^{\prime }, \zeta ^{\prime })}f(z)\overline{%
h(z- \zeta ^{\prime })}$

$=\underset{z\in W}{\sum }\frac{1}{\sqrt{\left| L\right| \left| L\cap
L'\right| }}\underset{ \zeta ^{\prime }\in L'}{\sum }%
\overline{\chi (z, \zeta ^{\prime })}f(z)\overline{\chi (z,- \zeta
^{\prime })h(z)}$

$=\underset{z\in W}{\sum }\frac{1}{\sqrt{\left| L\right| \left| L\cap
L'\right| }}\underset{ \zeta ^{\prime }\in L'}{\sum }%
f(z)\overline{h(z)}=\underset{z\in W}{\sum }\frac{\left| L'\right| 
}{\sqrt{\left| L\right| \left| L\cap L'\right| }}f(z)\overline{h(z)}
$

$=\underset{z\in W}{\sum }\frac{\left| L\right| }{\sqrt{\left| L^{\prime
}\right| \left| L\cap L'\right| }}f(z)\overline{h(z)}=\underset{%
z\in W}{\sum }\frac{1}{\sqrt{\left| L'\right| \left| L\cap
L'\right| }}\underset{ \zeta \in L}{\sum }f(z)\overline{h(z)}$

$=\underset{z\in W}{\sum }\frac{1}{\sqrt{\left| L'\right| \left|
L\cap L'\right| }}\underset{ \zeta \in L}{\sum }\chi
(z,- \zeta )f(z)\overline{\overline{\chi (z- \zeta , \zeta )}h(z)}$

$=\underset{w\in W}{\sum }\frac{1}{\sqrt{\left| L'\right| \left|
L\cap L'\right| }}\underset{ \zeta \in L}{\sum }f(w)\overline{%
\overline{\chi (w, \zeta )}h(w+ \zeta )}$

$=\left\langle f,\gamma _{L,L'}(h)\right\rangle $

This proves a).
Next, let us prove c).

$(\gamma _{L,L'}\circ \gamma _{L',L}f)(w)=\frac{1}{\sqrt{%
\left| L\right| \left| L\cap L'\right| }}\underset{ \zeta \in L}{%
\sum }\overline{\chi (w, \zeta )}(\gamma _{L',L}f)(w+ \zeta )$

$=\frac{1}{\left| L\right| \left| L\cap L'\right| }\underset{%
 \zeta \in L, \zeta ^{\prime }\in L'}{\sum }\overline{\chi
(w, \zeta )\chi (w+ \zeta , \zeta ^{\prime })}f(w+ \zeta
+ \zeta ^{\prime })$

$=\frac{1}{\left| L\right| \left| L\cap L'\right| }\underset{%
 \zeta \in L, \zeta ^{\prime }\in L'}{\sum }\overline{\chi
(w, \zeta )\chi (w+ \zeta , \zeta ^{\prime })}\chi (w+ \zeta
^{\prime }, \zeta )f(w+ \zeta ^{\prime })$

$=\frac{1}{\left| L\right| \left| L\cap L'\right| }\underset{
 \zeta ^{\prime }\in L'}{\sum }\left( \underset{ \zeta \in L}{%
\sum }\chi ( \zeta ^{\prime },2 \zeta )\right) \overline{\chi
(w, \zeta ^{\prime })}f(w+ \zeta ^{\prime })$

$=\frac{1}{\left| L\cap L'\right| }\underset{ \zeta ^{\prime
}\in L'\cap L}{\sum }\overline{\chi (w, \zeta ^{\prime })}%
f(w+ \zeta ^{\prime })$

$=f(w)$

Now, b) follows from a) and c).

Now we prove d).

$(\gamma _{L,L''}\circ \gamma _{L'',L^{\prime
}}\circ \gamma _{L',L}\ f)(w)=$

$\frac{\left( \left| L\right| ^{3}\left| L''\cap L\right|
\left| L'\cap L''\right| \right) }{\sqrt{\left|
L\cap L'\right| }}\underset{ \zeta \in L}{\sum }\underset{%
 \zeta ^{\prime }\in L'}{\sum }\underset{ \zeta \in L}{\sum }%
\overline{\chi (w, \zeta )\chi (w+ \zeta , \zeta '')\chi (w+ \zeta + \zeta '' , \zeta ^{\prime })} \cdot$
\mbox {$ f(w+ \zeta + \zeta '' + \zeta ^{\prime })=$}

$\frac{\left( \left| L\right| ^{3}\left| L''\cap L\right|
\left| L'\cap L''\right| \right) }{\sqrt{\left|
L\cap L'\right| }}\underset{ \zeta '' \in
L''}{\sum }\underset{ \zeta ^{\prime }\in L'}{%
\sum }\underset{\ }{(\underset{ \zeta \in L}{\sum }\chi ( \zeta
'' + \zeta ^{\prime },2 \zeta ))}\overline{\chi
(w, \zeta ^{\prime }+ \zeta '' )\overline{\chi
( \zeta '' , \zeta ^{\prime })}}\cdot $ \mbox{$f(w+ \zeta ^{\prime
\prime }+ \zeta ^{\prime }) = $}

$\frac{1}{\sqrt{\left| L\right| \left| L''\cap L\right|
\left| L'\cap L''\right| \left| L\cap L^{\prime
}\right| }}
 \underset{  \underset{ { \zeta'+ \zeta''  \in L}}{ \zeta'\in L', \zeta''
\in L''}} {\sum }\chi ( \zeta', \zeta'' )\overline{\chi (w, \zeta'+ \zeta'')}f(w+ \zeta'+ \zeta '')=$

$\sqrt{\frac{\left| L'\cap L''\right| }{\left|
L\right| \left| L''\cap L\right| \left| L\cap L^{\prime
}\right| }}\underset{ \zeta \in L\cap (L'+L'')}{
\sum }\chi ( \zeta ^{\prime }, \zeta '' )$

Finally, the last statement e) is straightforward.\medskip
\end{proof}
 
\subsubsection{The Weil representation of $G = SL(2, A) $ constructed by contraction of its Lagrangian bundle}
\label{ss lagrangian weil}
Once the Lagrangian bundle of  $G$ is constructed with its natural connection, we obtain the projective Weil representation of  $G$ by the contraction procedure presented in the previous paragraph.  Whether this projective representation may be ``corrected'' to afford a true representation of $G$ will depend on whether its  cocyle $c$ is cohomologically trivial or not, i.e. is a coboundary, which in turn depends essentially on the nature of the involutive ring $A$. 
We summarize our results so far in the following theorem, keeping the notations above. 

\begin{theorem}

Let $(S,\eta )$ be a self dual $A-$module. For each $L\in \mathcal{L},$ we
have a unitary projective representation $(V_L ;  \rho^L )$ of $G = SL_{\ast }(2,A)$,
that we call Weil representation of $G$ associated to  $L$, given by 
$V_L=\mathcal{E}_{L}$ and $\rho _{g}=\gamma _{L,gL}\circ \tau _{g}$. 

Its cocycle $c_L$  is defined
by 
$$c_L(g, h)= \sqrt{\frac{ | gh.L \cap g.L | }{| L\cap gh.L|
|g.L \cap L| | L | }} S_{W}(L; g.L, gh.L)  \ \ \  (g,h \in SL_{\ast}(2,A))$$
where  in general, for   $L, L', L''  \in  \mathcal L $,  \ $S_{W}(L;L',L'')$ denotes the geometric 
Gauss sum given by 

$$\underset{ \zeta \in L\cap (L' + L'')}{\sum }\chi ( \zeta ^{\prime }, \zeta '' )$$ 
where $ \zeta \in L$ $\cap (L'+L'')$ \ is written as 
$ \zeta'  + \zeta '' $ with $ \zeta ' \in L', \zeta '' \in L''.$
\end{theorem}
\begin{flushright} 
$ \square$
\end{flushright} 

For an explicit computation of the cocycle     
  see  \cite{PSAcr} for the finite full matrix ring case    and  \cite{JP} for  the locally profinite  case. 
  
  In the next section we will consider the case of a non semi-simple involutive ring.

\section{  Weil representations of $G = SL_{\ast }(2,A)$ for a non semi-simple $A$.}
\subsection{The involutive ring  $A_{m}$}
We  consider now the case of a commutative non semi-simple involutive ring $A$ with nilpotent radical, arising from the modular group algebra of the cyclic group  $C_p$ of order $p$ over $  \mathbb F_p$ (cf. section 2, example  iii.), which may also be regarded as the $p$-analogue of the real algebra of $p$ - jets in one variable, i. e. a truncated polynomial algebra over  $ \mathbb F_p.$     
 
  More generally, following   \cite{lucho},    let
  $k=\mathbb{F}_q$  be the finite field  with $q$ elements, $q$ odd and   $m$   a positive integer. Set

\[ A_{m}=k[x]/\left\langle x^{m}\right\rangle = \left\{
{\displaystyle\sum\limits_{i=0}^{m-1}}
a_{i}x^{i}:a_{i}\in k,x^{m}=0\right\}.
\]

We denote by $*$  the $k$-linear involution  on $A_m$ given by $x\mapsto -x$, which is, up to isomorphism, the unique involution on  $A_m$, besides the Identity, as proved in \cite{lucho}.

Regarding notations, in this section we will write  $A_m^s$ instead of  $(A_m)^s$ for the set of symmetric elements in  the involutive ring  $A_m.$

 \subsection{Bruhat Presentation of the group $G = SL_*(2,A_m)$.}
\label{section:presentation}
 
From now on, the involution $\ast$ on the ring $A_m$  will denote either the identity or the essentially unique non-trivial involution given by $x^* = -x$ on $A_m$.  In \cite{lucho} we have proved  the following two key lemmas.

\begin{lemma}\label{lemma 5}
Let $a$, $c$ be two elements in $A_m$ such that $a$ or $c$ is invertible and $a^*c=c^*a$. Then there is a symmetric element $s$ such that $a+sc$ is an invertible element in $A_m$. 
\end{lemma}
\begin{flushright}
$\square$
\end{flushright}
\begin{lemma}\label{lemma 6}
Let $a$, $c$ be two non-invertible symmetric elements in $A_m$. Then there is a symmetric invertible element $x$ in $A_m$ such that $a-x^{-1}$ and $b+x$ are symmetric invertible elements in  $A_m$. 
\end{lemma}
\begin{flushright}
$\square$
\end{flushright}

From the first lemma it follows that $A_m$ is weakly euclidean  and so, by theorem   \ref{eucligen}, the group $ G = SL_*(2,A_m)$ is generated by its Bruhat generators. Moreover,  the second lemma entails   that the $w$-length of any element of  $G$ is at most $2$:

\begin{proposition}\label{nilpotbruhatgen}
The group $G$ is generated by the set of  matrices $h(t)$,  $t \in A_m^{\times}$, $u(b)$, $b\in A_m^{s}$  and $w$.

 $h(a)=\left( 
\begin{array}{cc}
a & 0 \\ 
0 & (a^{\ast}) ^{-1}
\end{array}
\right), \ a \in  A^\times  $\  \ \ $\ w =\left( 
\begin{array}{cc}
0 & 1 \\ 
-1 & 0
\end{array}
\right) $ \ and \ \ $ u(s)=\left( 
\begin{array}{cc}
1 & s\\ 
0 & 1
\end{array}
\right), 
\\ \ s \in  A_s  $, which we call ``Bruhat generators'' for $G.$ \

 More precisely, if we put

\[
B =  \left\{  \left( \begin{array}{rr}
  a & b \\
  0 & d
  \end{array}\right)  \in SL_*(2,A_m) \right\},
\]
then $B$ is a subgroup of  $SL_*(2,A_m)$ and   
$SL_*(2,A_m)=B\cup BwB\cup BwBwB, $
so that the Bruhat length of  $SL_*(2,A_m)$  is $2$.
\end{proposition}

As in   \cite{lucho}  we can now easily prove:
 
\begin{theorem}\label{teorema bruhat nilpot}
The set of  matrices $h(t)$,  $t \in A_m^{\times}$, $u(b)$, $b\in A_m^{s}$ and $w$ of $SL_*(2,A_m)$ together with the  relations:  

\begin{enumerate}
\item $h(t_{1})h(t_{2})=  h(t_{1}t_{2})$,
\item $u(b_{1})u(b_{2}) = u(b_{1} + b_{2})$,
\item $h(t)u(b) =  u(tbt^{*})h(t)$,
\item $w^{2} =  h(-1)$,	
\item $wh(t) = h(t^{*^{-1}})w$,
\item $u(t)wu(t^{-1})wu(t)  = wh(-t^{-1})$,
\end{enumerate}
 give a  presentation of the group $SL_*(2,A_m)$, which we call the  Bruhat presentation of $SL_*(2,A_m)$.
\end{theorem}
   \subsection{Construction of a Weil Representation of  $G = SL_*(2,A_m)$  via generators and relations} \label{ss weil nilp gen rel}  
   
 We recall here the construction of the Weil representation of $SL_*(2,A_m)$ via generators and relations, given in  \cite{lucho}, which takes advantage of the Bruhat presentation of this group, as in  \cite{sa1} 
 
 To this end we consider the non-degenerate quadratic  $A_m$-module $(A_m, Q, B_Q)$, where  $Q : A_m  \rightarrow A_m $ is given by  $Q(t) = t^{\ast}t$ and $B_Q : A_m \times A_m \rightarrow A_m$ is such that $B_Q(t,s) = t^\ast s + ts^\ast$. Let $\tr$ be the linear form on $A_m$  defined by $\tr\left(\sum_{i=0}^{m-1}a_{i}x^{i}\right) = a_{m-1}$.

\begin{proposition} \label{proposition 14} Let $m$ be an odd number. Then
\begin{enumerate}
\item The form $\tr$ on $A_m$ is $k$-linear and it is invariant under the \mbox{involution $*$}.  \label{part 1}
\item If $*$ is the non-trivial involution, then the $k$-form $\tr \circ B_Q$ is a non-degenerate symmetric bilinear form.\label{part 2}
\end{enumerate}
\end{proposition}

We assume now that $m$ is odd and that $\psi$ is a non-trivial character of $k^+$. Set $\underline{\psi} = \psi\circ\tr$, so that   $\underline{\psi}$  is a non trivial character of    $A_m^+.$

Recall that the Gauss sum $S_{\underline{\psi} \circ Q}$ associated to any $A$-quadratic module $(M,Q,B_Q)$ and to $\underline{\psi}$ is defined by 

\[S_{\underline{\psi} \circ Q}(a) = \sum_{m\in M}\underline{\psi}(aQ(m))\,.\]

for all  $a \in  A. $
 The function  $\alpha$  from $A_m^\times \cap A_m^s $ to $\mathbb{C}^{\times}$ given by  $$  \alpha(a) = \frac{S_{\underline{\psi}\circ Q}(a)}{S_{\underline{\psi}\circ Q}(1)} \ \ \ \ \ \ \ \ \ (a \in A)$$ 
 is not in general a constant function as in the case  considered in  \cite{sa1}, where  $A = M(n, k), \ k$ a finite field of odd characteristic, endowed with the transpose mapping   (section 2, example i.)  and  $M$ has  dimension  $2rn$ as a $k$-vector space.
 
In \cite{lucho} we prove that the function  $\alpha$ coincides with the 
 sign character of the  group $A_m^{\times}\cap A_m^s$ of all symmetric invertible elements in  $A_m$.  With these notations, we have:

\begin{theorem} {\cite{lucho}} 
\label{Teorema8}
Let $W$   be the $\mathbb{C}$-vector space of all complex functions on $A_m$. Set
\begin{itemize}
	\item[i.]  $\rho(h(t))(f)(a) = \alpha(t)f(at)$,   
	\item[ii.]  $\rho(u(b))(f)(a) = \underline{\psi}(bQ(a))f(a)$,  $f\in W$ and  $a$,$b$ $\in A_m$,
	\item[iii.]  $\rho(w)(f)(a) = \frac{\alpha(-1)}{S_{\underline{\psi}\circ Q}(1)}\displaystyle\sum_{c\in A_m}\underline{\psi}(B_Q(a,c))f(c)$,  $f\in W$ and $a\in A_m$.
\end{itemize}
where $f \in W,  a \in A_m, t \in A_m^\times, b \in A_m^s $. These formulas define a linear representation $(W,\rho)$ of $G$, which we call  generalized Weil representation of $G$.
\end{theorem}

 \subsection{Geometric construction of a Weil representation of  $G = SL_*(2,A_m)$ via contraction of its Lagrangian bundle} \label{geom weil}
  
 We will apply now the general construction of a Lagrangian bundle for   $G = SL_*(2,A)$  to the case where    $A = A_m$.  
Keeping the notations of \ref{sslagrangian} we  take the $A_m$- module  $S$ to be   $A_m$ itself and
$\eta(a,b) = tr (a^\ast b) $  for all   $  a, b \in A_m $ where $tr$ denotes the trace mapping from   $A_m$ to  $k$ defined in \ref{ss weil nilp gen rel}.  
Then 
 $W= A_m\oplus A_m $ and  the symplectic bi-character $\chi
:W\times W\rightarrow \mathbb{U}$ is given by  $\chi  =  \psi \circ B$
 where  $B$ is the non degenerate symplectic form  on  $W$ given by
 $ B(v, w) = \eta (v_1, w_2) - \eta(v_2, w_1) $   \ \   for all  $v = (v_1, v_2), w = (w_1, w_2) \in W $.

\begin{example}
For $m = 3$ for instance, if  we write  
  
$v_i   =  v_i^{(0)} + v_i^{(1)}x + v_i^{(2)}x^2    \  \  \  \  \ (i = 1, 2) $

$ w_i   =  w_i^{(0)} + w_i^{(1)}x + w_i^{(2)}x^2    \  \  \   (i = 1, 2) $

then we find that 
$$  B(v,w) =   (v_1^{(0)}w_2^{(2)}  -  v_2^{(2)}w_1^{(0)}) + 
                       (v_1^{(2)}w_2^{(0)}  -  v_2^{(0)}w_1^{(2)})   + 
                       (v_2^{(1)}w_1^{(1)}  -  v_1^{(1)}w_2^{(1)}) ,                       $$

an orthogonal sum of  3  hyperbolic planes, so that  
$$ B  \cong  (y_1z_6 - y_6z_1)  +  (y_3z_4 - y_4z_3) + (y_5z_2 - y_2z_5).$$

\end{example}

 Notice now that for  $A _ m$ we have 
$ B  =   tr \circ \mathbb B, $
where  $tr$ denotes the trace from $A_m$   $\mathbb B$ is the   $A_m$-valued   anti-hermitian form on  $W$  defined by
$ \mathbb B(v,w) =   v_1^\ast w_2 - v_2^\ast w_1  $   \ \   for all  $v = (v_1, v_2), w = (w_1, w_2) \in W $. So, in fact, Lagrangians in  $W$ may be defined directly in terms of the  $A_m$-valued anti-hermitian form $\mathbb B$:
\begin{proposition}
An  $A$-submodule  $L$ of  $W$ is a Lagrangian  iff  $L$ is maximal totally isotropic with respect to the anti-hermitian form  $\mathbb B$, i.e.  $L$ coincides with its orthogonal relative to  $\mathbb B.$
 
\end{proposition}

\begin{proof}
If $L $ is a Lagrangian  then  
 $$ 0 =  B(u,va) = tr(\mathbb B(u,va)) = tr (\mathbb B(u,v)a)$$  for all  $ u, v \in L, a \in A. $ Since  tr is non degenerate it follows that     $B(u,v) = 0 $  \ for all   $ u, v \in L,$  so that  $L$ is totally isotropic for  $\mathbb B. $  Moreover, if  $w \in W $  is such that  $\mathbb B(u,w) = 0 $ for all   $u \in  L$ then also
 $ B(u,w) = tr (\mathbb  B(u,w)) = 0 $ for all $u \in  L$ so that $ w \in  L^\perp = L$.
 \end{proof}

Also notice that any Lagrangian $L$ is a usual Lagrangian  for the $k$-symplectic space  $(W, B)$, and so it has  k-dimension equal to  the k-dimension of   $S = A_m$, namely   $m.$
  
 Recall that  in \ref{ss lagrangian weil} we have shown how to construct a projective representation   $(V, \rho^L) $   of   $G$ 
  by contraction of the Lagrangian $G$-bundle $\mathcal H = (\mathcal E, p, \mathcal L, \tau)$ associated to  $G$ and  $S$ over a chosen base point  $L \in  \mathcal L$ along the connection  $\Gamma$.  The cocycle  $c_{L}$ of  $(V, \rho^L) $  is given by 
  $c_{L}(g, h)= \sqrt{\frac{ | gh.L \cap g.L | }{| L\cap gh.L|
|g.L \cap L| | L | }} S_{W}(L; g.L, gh.L)  $ \ \ \  $(g,h \in G)$
where  in general, for   $L, L', L''  \in  \mathcal L $,  \ $S_{W}(L;L',L'')$ denotes the geometric 
Gauss sum given by $\underset{ \zeta \in L\cap (L' + L'')}{\sum }\chi ( \zeta', \zeta '' )$ 
where $ \zeta \in L$ $\cap (L'+L'')$ \ is written as 
$ \zeta'+ \zeta '' $ with $ \zeta'
\in L', \zeta '' \in L''.$

\subsubsection*{Explicit formulas of the geometric Weil operators for the Bruhat generators of  $G$}

We calculate now the explicit form of the geometric Weil operators corresponding to the Bruhat generators of  $G,$  given by the contraction procedure \ref{geom weil}.   

Notice that to obtain nicer formulas we will take the natural (left) action of $G$ on its Lagrangian bundle to be induced from  right matrix multiplication   $w \mapsto  w.g$  in $W = A_m \oplus A_m $. Then the general notation  $gw$ means in fact  $ g.w^{-1}, $ so that, for instance     $ (\tau_g f)(w) = f(g^{-1}w ) = f(w.g) $ \  for 
$g \in G,  f \in \mathcal E_L, w \in W.$

We contract our Lagrangian bundle over the Lagrangian  $L_0 =  \langle (0,1) \rangle $ generated by $(0,1) \in W$. 
Notice that every function  $f \in  \mathcal E_{L_0} = V^{L_0} $ is completely determined by its values $ f(a,0)$ on the Lagrangian  $L_1= \langle (1,0) \rangle $, supplementary to $L_0$ in $W.$  We define then a linear isomorphism   $ \phi:  f \mapsto  f' $ from  $V_{L_0}$ to  $L^2(A)$ where  $f'(a) = f(a,0)$ for all $a \in A.$
We ``translate'' now the geometric Weil operators  $\rho^{L_0}_g =  \gamma_{L_0, g(L_0) } \circ \tau_g $, where $g$ is a Bruhat generator of  $G$,  into operators  $\sigma_g$ in $L^2(A)$ via $\phi.$  

Recall that in our case the connection isomorphisms  $ \gamma_{L', L}  \ \ \ (L, L' \in \mathcal L) $ are given by 
$$ (\gamma_{L', L} f)(u) =  \frac{1}{\sqrt{\left| L\right| \left| L\cap L'\right| }}   \underset {v \in L'} {\sum } \underline \psi(- u_1^\ast v_2 + u_2^\ast v_1)f(u+ v),  \ \  \   (f \in \mathcal E_L, u = (u_1, u_2) \in W) $$

where we have written    $ \underline \psi = \psi \circ tr  $.

\begin{proposition} \label{explicit}
With the above notations, the operators  $\sigma_g$ in $L^2(A)$, corresponding via the isomorphism $\phi$ to the geometric Weil operators   $\rho^{L_0}_g$, where  $g$ is a Bruhat generator of $G,$  are given by
\begin{enumerate}
\item [i.]  $( \sigma_{h(a)}f')(c) =  f'(ac)        \ \ \ (a \in A^\times, f' \in L^2(A), c \in A). $

\item [ii.]   $  ( \sigma_{u(b)}f')(c) = \underline \psi (bcc^\ast)f'(c)        \ \ \ (b \in A^s, f' \in L^2(A), c \in A). $

\item [iii.]    $  ( \sigma_w f')(c) = q^{-m/2}   \underset {a \in A} {\sum } \underline \psi( 2c^\ast a)  f'(a),$
for all  $ f' \in L^2(A), c \in A. $
\end{enumerate}

\end{proposition}
\begin{proof}

To prove i. notice that 
$ \ \  (\tau_{h(a)} f)(c, 0) = f(ac, 0),      $ \ \ so that 
$$  ( \sigma_{h(a)}f')(c) =  f'(ac)        \ \ \ (a \in A^\times, f' \in L^2(A), c \in A). $$          
To prove ii. we just check that for  $ b\in A^s$ and  $f \in  \mathcal E_{L_0}$ we have 
$$  (\tau_{u(b)} f)(c, 0) = f(c, cb) = f( (c,0) + (0,cb)) = \underline \psi (bcc^\ast)f(c,0)    $$
so  that
$$  ( \sigma_{u(b)}f')(c) = \underline \psi (bcc^\ast)f'(c)        \ \ \ (b \in A^s, f' \in L^2(A), c \in A). $$
To prove iii., we calculate   
$$(\rho^{L_0}_w f)(c,0) = (\gamma_{L_0, w(L_0) } \circ \tau_w f)(c,0) =
      q^{-m/2}   \underset {a \in A} {\sum } \underline \psi(- c^\ast a) (\tau_wf)(c,a) = $$
  $$  =  q^{-m/2}   \underset {a \in A} {\sum } \underline \psi(- c^\ast a) f(-a,c)  
  =  q^{-m/2}   \underset {a \in A} {\sum } \underline \psi(- c^\ast a) f((-a,0)+ (0,c)) =   $$
  $$ = q^{-m/2}   \underset {a \in A} {\sum } \underline \psi(- c^\ast a)\underline \psi (-a^\ast c) f(-a,0) = 
 q^{-m/2}   \underset {a \in A} {\sum } \underline \psi( 2c^\ast a)  f(a,0), $$ 
 so that 
     
$$  ( \sigma_w f')(c) = q^{-m/2}   \underset {a \in A} {\sum } \underline \psi( 2c^\ast a)  f'(a), $$
for all  $ f' \in L^2(A), c \in A. $

\end{proof}
Now we can compare easily on Bruhat generators, our geometric Weil operators with the Weil operators given in \cite{lucho}.
 
\begin{theorem}
With the notations of prop. \ref{explicit}, the geometric Weil operators $\sigma_g,  \  g $ a Bruhat generator of $G,$ compare as follows with the corresponding Weil operators  $\rho(g)$ in \cite {lucho}:

\begin{enumerate}
\item [i.]   $ \rho(h(a)) = \alpha(a) \sigma_{h(a)}     \ \ \  (a \in A);$
\item [ii.]    $  \rho(u(b)) =   \sigma_{u(b)}     \ \ \  (b \in A^s);$
\item [iii.]      $ \rho(w) =   \omega(\psi, q) \sigma_w $
where   $ \alpha $ denotes the sign character of  $A$, which coincides with the Legendre symbol on $k = \mathbb F_q $ and $ \omega(\psi, q)$ denotes the  fourth root of unity such that   
 $ \underset {t \in k}  \sum \psi (t^2 )  =    \omega(\psi, q) \sqrt q,$ whose square is  $ \alpha(-1). $
 \end{enumerate}
 \end{theorem}
\begin{flushright}
$\square$
\end{flushright}

\begin{corollary}
 The 2-cocycle of our geometric projective  Weil representation   $(V_{L_0}, \rho^{L_0})$ is cohomologically trivial. More precisely it is the  coboundary of the 1-cocyle  $\delta$ on $G$ defined as follows, with the help of the Bruhat presentation of  $G$:
 \begin{enumerate}
 \item [i.]   $  \delta ( h(a)u(b)) = \alpha(a)     \ \ \ \   (a \in A^\times,  b \in A^s) $
  \item [ii.]   $  \delta ( h(a)u(b)wu(c)) = \alpha(a)\omega(\psi, q)     \ \ \   (a \in A^\times,  b, c \in A^s) $
  \item [ii.]   $  \delta ( h(a)u(b)wu(c)wu(d)) = \alpha(-a)     \ \ \   (a \in A^\times,  b, c, d \in A^s) $
  \end{enumerate}
  \end{corollary}

\end{document}